\documentclass[bj,authoryear]{imsart}

\usepackage{bbm}

\startlocaldefs
\newcommand{\ddr}{\mathrm{d}}

\def\E{\mathbb{E}}

\theoremstyle{plain}

\newtheorem{theorem}{Theorem}[section]
\newtheorem{lemma}[theorem]{Lemma}
\newtheorem{corollary}{Corollary}
\theoremstyle{remark}
\newtheorem{remark}[theorem]{Remark}
\newtheorem{definition}[theorem]{Definition}
\newtheorem*{example}{Example}

\endlocaldefs

\begin{document}

\begin{frontmatter}
\title{Extremal shot noise processes and 
random cutout sets}
\runtitle{Extremal shot noise processes and  random cutout sets}

\begin{aug}
\author[A]{\inits{F.}\fnms{Cl\'ement}~\snm{Foucart}\ead[label=e1]{foucart@math.univ-paris13.fr}}\orcid{0000-0002-6750-8795}
\author[B]{\inits{L.}\fnms{Linglong}~\snm{Yuan}\ead[label=e2]{linglong.yuan@liverpool.ac.uk}}\orcid{0000-0002-7851-1631}
\address[A]{CMAP, Ecole Polytechnique, Palaiseau, France and \\ 
\ LAGA, Institut Galil\'ee, Universit\'e Sorbonne Paris Nord\\ Villetaneuse, France\printead[presep={,\ }]{e1}}
\address[B]{Department of Mathematical Sciences\\
University of Liverpool, United Kingdom \printead[presep={,\ }]{e2}}
\end{aug}

\begin{abstract}
We study some fundamental properties, such as the transience, the recurrence, the first passage times and the zero set of a certain type of sawtooth Markov processes, called extremal shot noise processes. The sets of zeros of the latter are Mandelbrot’s random cutout sets, i.e. the sets obtained after placing Poisson random  covering intervals on the positive half-line. Based on this connection, we provide a new proof of Fitzsimmons-Fristedt-Shepp Theorem  which characterizes the random cutout sets.
\end{abstract}

\end{frontmatter}
\section{Introduction}\label{CBI}

Extremal shot noise processes (ESNs) first appeared in the eighties in the framework of applied stochastic geometry and of random sets for modelling extremes in a spatial setting, see \cite{zbMATH03902440} (page  470) and \cite{zbMATH00796709}.  They have been then reintroduced in a more general setting by \cite{Dombry} who has studied some of their properties and shed light on their connection with max-stable random fields. We work here in the setting of one-dimensional Markov processes.

Let $\mathcal{N}:=\sum_{s\geq 0}\delta_{(s,\xi_s)}$ be a Poisson point process (PPP)
on $[0,\infty)\times (0,\infty)$ with intensity $\lambda \times \mu$ with $\mu$ a Borel measure on $(0,\infty)$ and $\lambda$ the Lebesgue measure. We denote the tail of $\mu$ by $\bar{\mu}(x)=\mu([x,\infty))$ and suppose $\bar{\mu}(x)<\infty$ for all $x>0$. Denote by $(a)_+:=\max(a,0)$, the positive part of any real number $a$. Let $b\in \mathbb{R}$.

\begin{definition}\label{ESNdef}
We call standard $\mathrm{ESN}(b,\mu)$ the process $(M(s),s\geq 0)$ valued in $[0,\infty)$ and obtained from $\mathcal{N}$ as follows:
\[M(t):=\sup_{0\leq s\leq t}\big(\xi_s-b(t-s)\big)_+.\]
\end{definition}

\noindent When $b=0$, the process $M$ is a classical extremal process. Those processes have been studied by \cite{dwass1966} and \cite{APR:10159662}. We refer also to Resnick's book \cite[Chapter 4.3]{Res87} and the references therein. When $b\neq 0$,  the contribution of any atom to the process is affected by its age, which is the so-called shot noise structure.

Extremal shot noise processes as defined above form a certain class of sawtooth processes, in the sense that they evolve linearly or stay constant between their jumps. Such processes are known to play an important role in the theory of Markov processes, see for instance \cite[page 49]{zbMATH01349990}. When the intensity measure $\mu$ of the Poisson point process $\mathcal{N}$ is finite, $M$ is a piecewise deterministic Markov process.  
We refer the reader for instance to \cite{Zbl0565.60070, Zbl0780.60002} for a thorough study of this class of processes. 
In a close spirit of ESNs, certain processes with jumps of finite intensity and linear release, have been studied by \cite{zbMATH05033694}, \cite{Zbl0994.60092} and \cite{zbMATH05918584}.

It turns out that many natural problems can be solved with closed-form solutions for ESNs. 
Their finite dimensional laws, their semigroup and their long-term behavior are for instance obtained in Theorem \ref{thmfddsemigroupESN}.   The generator is studied in Theorem \ref{thmgeneratorESN}, and last but not least, the Laplace transform of their first-passage times, also available explicitly in terms of $b$ and $\mu$, is given in Theorem \ref{thmfirstpassagetimeM}.


Random cutout sets were introduced by \cite{Mandelbrot}. They are defined as the sets of real numbers left uncovered by  Poisson random covering intervals on the positive half line. Namely, the random cutout set based on the PPP $\mathcal{N}$ is given by 
\begin{equation}\label{randomcutoutset} \mathcal{R}:=[0,\infty)-\bigcup_{s\geq 0}(s,s+\xi_s).\end{equation}
Those random sets are at the core of the theory of random coverings, see \cite{zbMATH04203384,zbMATH01424076}. Some of their multifractal properties have been studied by  \cite{zbMATH02093931, zbMATH02196675}. They also appear in many other contexts. We refer for instance to \cite{zbMATH00029121,zbMATH00812718} where they are used for studying the existence of increase times in Lévy processes and the non-differentiability of their sample paths.  They also play a crucial role in the study of zero sets of certain processes, see e.g.\ \cite{zbMATH06347448},  \cite{zbMATH05678697} and  \cite{foucart2014}. Some random sets with closely related constructions are  studied in \cite{zbMATH06441328} and \cite{zbMATH01901511}.

The question of when the random set $\mathcal{R}$ is almost surely reduced to the singleton $\{0\}$, that is to say when the whole open half-line is covered, was asked in \cite{Mandelbrot} and a necessary and sufficient condition was found by \cite{zbMATH03378272}. A definitive answer on how to characterize the law of the random cutout set is given by a Theorem of \cite{Fitzsimmons}, with the potential measure of $\mathcal{R}$ given explicitly. Their arguments were based on approximations of $\mathcal{R}$ by intersections of regenerative sets, see also \cite[Chapter 7]{subordinatorbertoin} for those notions. 

A striking feature of the ESN process lies in its simple connection with the random cutout set associated to $\mathcal{N}$. We shall see that the random set $\mathcal{R}$ in \eqref{randomcutoutset} coincides with the closure of the zero set of a standard $\mathrm{ESN}(1,\mu)$ process. The main contribution of the paper is a new proof of Fitzsimmons-Fristedt-Shepp's Theorem, see Corollary \ref{corollaryrandomcutoutset}, based on this connection and on classical arguments of potential theory of Markov processes. Most important properties of $\mathcal{R}$, such as the regenerative property (see e.g. \cite{Maisonneuve} for this notion) and the fact that it is a perfect set (i.e. it has no isolated point), will also directly follow from this representation.\\ 

\noindent \textbf{Notation.} Let $\mathbb{R}_+:=[0,\infty)$. For any subset $A\subset \mathbb{R}_+$, we denote its closure by $\bar{A}$. For any $x,y\in \mathbb{R}$, we denote by $x\wedge y$ and $x\vee y$  the minimum and the maximum of $x$ and $y$.  In any integral $\int_a^b$, we adhere to the convention that the lower delimiter a is excluded from the integration, while the upper delimiter $b$ is included (except for $b=\infty$). We denote by $C_0([0,\infty))$ the space of continuous functions vanishing at $\infty$, and by $\|f\|_{\infty}$ the supremum norm of $f$. 
We set $C^{1,0}([0,\infty))$ the space of continuously differentiable functions vanishing at $\infty$ and whose derivative vanishes at $\infty$. For any function $f$, we denote by $f_{|[a,b)}$ the restriction of $f$ on the interval $[a,b)$. The limit inferior  and superior of a function $f$ are denoted respectively by $\liminf f$ and $\limsup f$. 
For any event $A$, we denote by $A^c$ the complementary event. Lastly, $X\overset{\text{law}}{=}Y$ means that the random variables $X$ and $Y$ have the same law. 




\section{Extremal shot noise processes as Markov processes}\label{secESN}
%

We first collect  some basic observations from Definition \ref{ESNdef}. Let $M$ be a standard $\mathrm{ESN}(b,\mu)$ process. The process $M$ takes only nonnegative values, starts from $0$ and has clearly c\`adl\`ag paths. 

In the case $b\leq 0$ (i.e. the slopes are nonnegative) the process has almost-surely non-decreasing sample paths which go towards $\infty$. When $b<0$, paths are increasing, and they are not monotonic when $b>0$, see Figure \ref{behaviorESN} below. Note also that by construction,  for any $b\in \mathbb{R}$, we have $M(t)\geq (-bt)_+\geq 0$ for all $t\geq 0$, $\mathbb{P}$-a.s..  
We shall mainly focus on the case $b\neq 0$ in this article. 

The Markov property of the Poisson point process $\mathcal{N}$ and the fact that the ``response function" $s\mapsto -b(t-s)$ in the shot noise structure is linear in time will imply that the standard ESN process $(M(t),t\geq 0)$ is a time-homogeneous Markov process. In particular there exists a family of probability distributions $(\mathbb{P}_x,x\in \mathbb{R}_+)$ on the space of non-negative c\`adl\`ag paths such that $\mathbb{P}_x$ is the law of the process $M$ started with initial value $M(0)=x$. The probability law $\mathbb{P}_x$ can be constructed on the same probability space as $\mathcal{N}$
by adjoining a point $(0,x)$ to the PPP $\mathcal{N}$. Namely set $\mathcal{N}^x:=\delta_{(0,x)}+\mathcal{N}$ and define 
\begin{equation}\label{Mx}
M^x(t)=\sup_{0\leq s\leq t}\left\{ \big(\xi_s-b(t-s)\big)_+: (s,\xi_s) \text{ is an atom of }  \mathcal{N}^x \right\}=(x-bt)_+\vee M^0(t),
\end{equation}
where $M^0$ is the standard $\mathrm{ESN}(b,\mu)$ process. 

We see from \eqref{Mx} that the process will leave $x$ along $(x-bt)_+$ until it encounters the first atom of $\mathcal{N}$ satisfying $\xi>(x-bt)_+$ and jumps there. If the process is able to reach the boundary $0$, then it stays at $0$ until the next atom of $\mathcal{N}$. 

In particular, any point $x>0$ is instantaneous (it is left immediately) and  is irregular for itself (the process does not return to it immediately). Indeed, when $b<0$, the paths being increasing, the point $x$ will not be reached again. When $b>0$, the process may only return to $x$ by firstly getting back above it and secondly reaching it by linear decay. Since by assumption $\bar{\mu}(x)<\infty$, returning to $x>0$ can occur only after a strictly positive time a.s..
\begin{figure}[htbp]
\centering \noindent
\includegraphics[height=.14 \textheight]{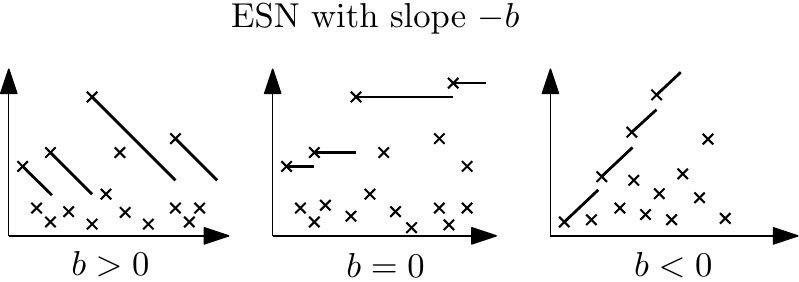}
\caption{Sample paths of an $\mathrm{ESN}(b,\mu)$.}
\label{behaviorESN}
\end{figure}
\vspace*{-5mm}

When $b>0$ and $\bar{\mu}(0)<\infty$, it is clear that the point $0$ will be reached with positive probability. This is furthermore a holding point. Indeed the process started at $0$ stays at $0$ for an exponential time with parameter $\bar{\mu}(0)$, the rate at which a new atom of $\mathcal{N}$ is encountered.  When $\bar{\mu}(0)=\infty$, the random set $\{0\leq s\leq t: \xi_s\in (0,1)\}$  is almost surely a dense subset of $[0,t]$ and the standard process $M$ (recall that it is the process starting with initial value $0$) makes immediately a positive jump a.s.. 

The only point at which the process may have a non trivial behavior is thus $0$. In the case $b>0$, a natural question is whether or not the negative slopes are strong enough for the paths to be able to reach $0$ when $\bar{\mu}(0)=\infty$. We may also wonder if the process can be transient. Before tackling this problem, see Theorem \ref{thmfirstpassagetimeM}, we gather in our first theorem fundamental properties of ESNs, including for instance their Markov property. 



\begin{theorem}[Finite dimensional laws, semigroup and stationary distribution]\label{thmfddsemigroupESN} Let $b\in \mathbb{R}$, $\mu$ be a measure on $(0,\infty)$ such that $\bar{\mu}(x)<\infty$ for all $x>0$, and $M$ be an $\mathrm{ESN}(b,\mu)$. 
\begin{enumerate}
    \item Let $n\geq 1$ and $0=s_0<s_1<s_2<\cdots<s_n$.  For any $u_1,\cdots, u_n\in \mathbb{R}_+$. 
\begin{align}
&\mathbb{P}_0(M(s_1)\leq u_1,\cdots, M(s_n)\leq u_n) \nonumber \\
&=\exp\left( -\sum_{i=1}^{n}\int_{s_{i-1}}^{s_{i}}\bar{\mu}\left(\bigwedge_{j=i}^{n}\Big(u_j+b(s_j-t)\Big)\right)\ddr t\right)\prod_{i=1}^n\mathbbm{1}_{\{u_i\geq (-bs_i)_+\}}.
\end{align}
In particular the one-dimensional law at time $s$ has the following cumulative distribution function: for any $u\in\mathbb{R}$,
\begin{align}
F^0_s(u):=\mathbb{P}_0(M(s)\leq u)&=\exp\left( -\int_{0}^{s}\bar{\mu}\left(u+b(s-t)\right)\ddr t\right)\mathbbm{1}_{\{u\geq (-bs)_+\}}\label{CDFM0allb}\\
&=\exp\left( -\frac{1}{b}\int_{u}^{u+bs}\bar{\mu}(y)\ddr y\right)\mathbbm{1}_{\{u\geq (-bs)_+\}} \text{ if } b\neq 0.\label{CDFM0bneq0}
\end{align}
 \item Let $x\in [0,\infty)$ and $t\geq 0$. For any $u\in \mathbb{R}$, $$F^x_t(u):=\mathbb{P}_x(M(t)\leq u)= F^0_t(u)\mathbbm{1}_{\{u\geq (x-bt)_+\}}.$$ 
\item The process $M$ is a Markov process with Feller property, i.e. its semigroup $(P_t)$ satisfies
 \begin{enumerate}
  \item $P_tC_0([0,\infty))\subset C_0([0,\infty))$,
   \item For any $f\in C_0([0,\infty))$, $P_tf\underset{t\rightarrow 0}{\longrightarrow} f$ uniformly.
 \end{enumerate}

\item  Assume $b\ge 0$. The following equivalence holds \[\forall s>0,  \mathbb{P}_0(M(s)=0)>0 \text{ if and only if } \int_0^1 \bar{\mu}(v)\ddr v<\infty.\]
Moreover, in this case  the Lebesgue measure of the zero-set of $M$, $\mathcal{Z}:=\{t>0: M(t)=0\}$, is strictly positive a.s. (and the boundary $0$ is said to be sticky). \smallskip
\item Assume $b>0$.  When $\int_1^\infty \bar{\mu}(u)\ddr u=\infty$, $M(s)\rightarrow \infty$ in $\mathbb{P}_x$-probability as $s$ goes to $\infty$, for all $x\geq 0$. When $\int_1^\infty \bar{\mu}(u)\ddr u<\infty$, the process $M$ admits a non-degenerate stationary distribution whose cumulative distribution function is given by 
\[\pi([0,u])=\exp\left(-\frac{1}{b}\int_u^\infty \bar{\mu}(v)\ddr v\right), \text{ for any } u\geq 0.\]
Moreover, when $\int_1^\infty \bar{\mu}(u)\ddr u<\infty$, for all $x\geq 0$,
the probability measure $\mathbb{P}_x(M(t)\in \cdot)$ converges in total variation distance towards  $\pi$. We have furthermore for all $t>x/b$,
\begin{equation}\label{dtvineq}
d_{\mathrm{TV}}\big(\mathbb{P}^{x}(M(t)\in \cdot),\pi(\cdot)\big)\leq \frac{1}{b}\int_{bt}^{\infty}\bar{\mu}(u)\ddr u,
\end{equation} 
where we have denoted by $d_{\mathrm{TV}}$ the total variation distance.
 \end{enumerate}
\end{theorem}
We now study the infinitesimal generator of $M$. 

\begin{theorem}[Generator of ESN]\label{thmgeneratorESN}  Let $b\in \mathbb{R}$. Denote by $(\mathcal A,\mathcal{D}(\mathcal{A}))$ the generator of the $\mathrm{ESN}(b,\mu)$ process and its domain. Set 
\begin{align}
\mathcal{D}_0&:=\{f \in C^{1,0}([0,\infty)): \exists \epsilon>0, \text{such that}\  f_{|[0, \epsilon]} \text{ is constant}\},
\label{eqn:ddpositive}
\end{align}
and
\begin{equation}\label{eqn:dd}
\mathcal{D}_1:=\left\{f \in C^{1,0}([0,\infty)):f'(0)=0 \text{ and }\int_0^1|f'(v)|\bar{\mu}(v)\ddr v<\infty\right\}.
\end{equation} 
Then \[\mathcal{D}_0\subset \mathcal{D}_1\subset \mathcal{D}(\mathcal{A}).\]
Moreover $\mathcal{A}$ acts on $\mathcal{D}_1$ as follows: 
\begin{equation}\label{eqn:loggenerator}
    \mathcal Af(x)=\int_x^\infty \Big(f(y)-f(x)\Big)\mu(\ddr y)-bf'(x), \text{ for any } x\geq 0.
    \end{equation}
When $b>0$, $\mathcal{D}_0$ and $\mathcal{D}_1$ are cores for $\mathcal{A}$. When $b\leq 0$ and $\int_0^1\bar{\mu}(v)\ddr v=\infty$, $\mathcal{D}_1$ is a core for $\mathcal{A}$.
\end{theorem}

\begin{remark}\label{remarkgenerator}
A simple use of Fubini-Lebesgue theorem gives the following alternative expression for the generator $\mathcal{A}$. For any $f\in \mathcal{D}_1$ and any $x\geq 0$: 
\begin{equation}\label{alternativeform} \mathcal{A}f(x)=\int_x^{\infty}\bar{\mu}(v)f'(v)\ddr v-bf'(x).
\end{equation}
Notice that $f\in \mathcal{D}_1$ entails $\int_0^1 \bar{\mu}(v)f'(v)\ddr v<\infty$ which ensures that $\mathcal{A}f(0)$ is well-defined.
\end{remark}
\begin{remark}\label{remarksinglecore} When $\int_0^1 \bar{\mu}(x)\ddr x=\infty$ the set $\mathcal{D}_1$ is a core of the $\mathrm{ESN}(b,\mu)$ for any $b\in \mathbb{R}$. Moreover, when $\int_0^1 \bar{\mu}(x)\ddr x=\infty$, the condition $\int_0^1 \bar{\mu}(v)f'(v)\ddr v<\infty$ supersedes $f'(0)=0$. In the case $b\leq 0$ and $\int_0^1\bar{\mu}(v)\ddr v<\infty$, we  only have been able to show that $\mathcal{D}_1$ is a subset of the domain. Finding a core in this case does not seem to follow easily from our approach. 
\end{remark}
In the next theorem, the first passage time of $M$ below any level is studied and the questions whether the process is recurrent or transient and if $0$ is accessible are addressed. For any $a\in [0,\infty)$, we set $\sigma_a:=\inf\{s\geq 0:M(s)\leq a\}$.
\begin{theorem}[First passage times and transience/recurrence]\label{thmfirstpassagetimeM} Let $M$ be an $\mathrm{ESN}(b,\mu)$ with $b>0$. 
\begin{enumerate}
\item Let $\theta>0$. Define for any $x>0$, 
\begin{equation}\label{thetainvariantfunction}
f_{\theta}(x):=\int_{x}^{\infty}e^{-\frac{\theta}{b}s}\exp{\left(\frac{1}{b}\int_{s}^{1}\bar{\mu}(u)\ddr u\right)}\ddr s.
\end{equation}
One has $f_\theta(x)<\infty$ for all $x>0$. For any $x>a>0,$
the Laplace transform of $\sigma_a$ is given by
\[\mathbb{E}_x[e^{-\theta \sigma_a}]=\frac{f_\theta(x)}{f_\theta(a)}.\]
\item  Set \begin{equation}\label{I} \mathcal{I}:=\int^{\infty}_1\exp\left(\frac{1}{b}\int_{s}^1\bar{\mu}(v)\ddr v\right)\ddr s.
\end{equation} We have the following dichotomy: \begin{itemize}

\item If $\mathcal{I}=\infty$, then $M$ is recurrent (i.e. it returns almost surely to any point $a>0$). 

\hspace*{3mm} Moreover,   
\begin{itemize}
\item in the case $\int^\infty_1 \bar{\mu}(v)\ddr v=\infty$,  $M$ is null recurrent,
\item in the case $\int^\infty_1 \bar{\mu}(v)\ddr v<\infty$,  $M$ is positive recurrent.
\end{itemize}
\item If $\mathcal{I}<\infty$, then $M$ is transient (i.e. $M(s)\rightarrow \infty$ a.s.)
\end{itemize}

\item Set \begin{equation}\label{J}\mathcal{J}:=\int_{0}^1\exp\left(\frac{1}{b}\int_{s}^1\bar{\mu}(v)\ddr v\right)\ddr s.
\end{equation} We have the following dichotomy: \begin{itemize}
\item If $\mathcal{J}=\infty$ then $0$ is inaccessible (i.e. $M(s)>0$ for all $s>0$ almost surely). 
\item If $\mathcal{J}<\infty$ then $0$ is accessible (i.e. $M(s)=0$ for some $s>0$ with positive probability).
\end{itemize}
\end{enumerate}
\end{theorem}
Finally, we identify the local time at $0$ of $M$ when $\mathcal{J}<\infty$.
\begin{theorem}[Inverse local time]\label{thminverselocaltime} Let $M$ be the standard ESN process. Assume $\mathcal{J}<\infty$. The point $0$ is regular for itself and the inverse of the local time at $0$ of $M$ is a subordinator $(\tau_x, 0\leq x< \zeta)$ (with possibly a finite lifetime $\zeta$) whose Laplace exponent is $\varphi:\theta\mapsto f_1(0)/f_\theta(0)$, where we have set  $f_\theta(0):=\underset{x\rightarrow 0+}{\lim} f_\theta(x)$ for any $\theta>0$.

Moreover, one has 
\[\overline{\{t\geq 0: M(t)=0\}}=\overline{\{\tau_x: 0\leq x<\zeta\}} \text{ a.s..}\]
\end{theorem}

The next lemma establishes the connection between extremal shot noise processes and random cutout sets. 

\begin{lemma}[Form of the zero set of ESN]\label{lemmazeroset} Let $b>0$. The closure of the zero set of the standard $\mathrm{ESN}(b,\mu)$ process $(M(t),t\geq 0)$, is of the following form
\[\overline{\{t\geq 0: M(t)=0\}}=[0,\infty)-\bigcup_{s\geq 0}(s,s+\xi_s/b) \text{ a.s..}\]
\end{lemma}

Theorem \ref{thminverselocaltime} and Lemma \ref{lemmazeroset} have for direct corollary, the theorem of \cite{Fitzsimmons},  which characterizes the random cutout set. 
\begin{corollary}[Theorem 1 in \cite{Fitzsimmons}]\label{corollaryrandomcutoutset}  Let $\mu$ be a measure on $(0,\infty)$, such that $\bar{\mu}(x)<\infty$ for any $x>0$, and $\mathcal{N}:=\sum_{s\geq 0}\delta_{(s,\xi_s)}$ be a PPP on $[0,\infty)\times (0,\infty)$ with intensity $\lambda \times \mu$.  Consider the random cutout set \[\mathcal{R}:=[0,\infty)-\bigcup_{s\geq 0}(s,s+\xi_s),\]
we have that
\begin{enumerate}
\item $\mathcal{R}=\{0\}$ a.s. if and only if $\int_0^1\exp\left(\int_s^1\bar{\mu}(v)\ddr v\right)\ddr s=\infty$ (Shepp's criterion, see \cite{zbMATH03378272}).
\item When $\int_0^1\exp\left(\int_s^1\bar{\mu}(v)\ddr v\right)\ddr s<\infty$, $$\mathcal{R}\overset{\mathrm{law}}{=}\overline{\{\tau_x: 0\leq x< \zeta\}},$$ 
where $(\tau_x, 0\leq x< \zeta)$ is a subordinator with lifetime $\zeta\in (0,\infty]$ and  Laplace exponent $\theta\mapsto f_1(0)/f_\theta(0)$. The random set $\mathcal{R}$ is regenerative and is a.s. perfect (the latter means that it has no isolated point). 
\item Furthermore, $\mathcal{R}$ is bounded a.s. if and only if $\int_1^\infty\exp\left(\int_s^1 \bar{\mu}(s)\ddr s\right) \ddr s<\infty$. It has positive Lebesgue measure a.s. if and only if $\int_0^1\bar{\mu}(s)\ddr s<\infty$.  

\end{enumerate}

\end{corollary}

\begin{remark} ESN processes satisfy the property of max-infinite divisibility, see \cite[Proposition 2.3]{Dombry}, that is to say,
\begin{equation}\label{mid} (M(t),t\geq 0)\overset{\mathrm{law}}{=}\left(\vee_{i=1}^{n} M_i(t), t\geq 0\right),\end{equation} where  $M$ is a standard $\mathrm{ESN}(b,\mu)$ process and the processes $(M_i, i=1,\cdots, n)$ are i.i.d.\  standard $\mathrm{ESN}(b,\frac{1}{n}\mu)$ processes. The identity \eqref{mid} is a direct consequence of the superposition theorem of Poisson point processes.  In terms of the zero-sets, this entails
$$\mathcal{Z}:=\{t\geq 0: M(t)=0\}\overset{\text{law}}{=}\{t\geq 0: \vee_{i=1}^{n}M_i(t)=0\}=\cap_{i=1}^n \mathcal{Z}_i,$$
with $\mathcal{Z}_i:=\{t\geq 0: M_i(t)=0\}$ for $i=1,\cdots, n$. We recover here the fact that the random cutout sets are infinitely divisible for the intersection, see \cite[Section 5]{Fitzsimmons} and \cite{zbMATH07605707}.
\end{remark}
Many explicit examples can be designed by choosing specific tails for the measure $\mu$, see for instance \cite{Fitzsimmons}. 
We first shed light on the ESN process whose inverse local time at $0$ is a stable subordinator and then give some explicit examples of stationary distributions.


\begin{example}[Selfsimilar ESN]\label{selfsimilarESN} Let $b\in \mathbb{R}_+$ and $c>0$. Assume $\bar{\mu}(x)=\frac{c}{x}$ for all $x>0$ and let $M$ be an $\mathrm{ESN}(b,\mu)$. Then
\begin{enumerate}
\item For any $x\geq 0$, $t\geq 0$, 
\[F_t^x(u):=\mathbb{P}_x(M(t)\leq u)=\left(\frac{1}{1+\frac{bt}{u}}\right)^{c/b}\mathbbm{1}_{\{u\geq (x-bt)_+\}}.\]
Moreover, for any $k>0$ and $u\geq 0$, $F_{t/k}^x(u/k)=F_{t}^{kx}(u)$ so that $(kM^x(t/k),t\geq 0)$ has the same law as $(M^{kx}(t),t\geq 0)$, i.e. $M$ is selfsimilar with index $1$.
\item The finite-dimensional law of $M$ satisfies for any $u_1,\cdots, u_n\in \mathbb{R}$, 
\begin{align}\label{finitedimexampleESN}
&\mathbb{P}_0(M(s_1)\leq u_1,\cdots M(s_n)\leq u_n)=\prod_{i=1}^{n}\bigwedge_{j=i}^{n}\left(\frac{u_j+b(s_j-s_{i-1})}{u_j+b(s_j-s_{i})}\right)^{\frac{c}{b}}\mathbbm{1}_{\{u_i\geq (-bs_i)_+\}},
\end{align}
where $0=s_0< s_1< \cdots < s_n$.
\item When $b>0$, Theorem \ref{thmfirstpassagetimeM}-(2) ensures that when $c/b>1$, $M$ is transient, otherwise it is null recurrent. Furthermore, by applying Theorem \ref{thmfirstpassagetimeM}-(3), we see that $0$ is accessible for $M$ if and only if $c/b<1$. In this case, $f_\theta(0)=\theta^{c/b-1}$ for all $\theta>0$ and by Theorem \ref{thminverselocaltime}, the inverse local time is a stable subordinator with index $1-c/b$. 
\end{enumerate} 
\end{example}
\begin{remark}\label{remGWI} The selfsimilar $\mathrm{ESN}(b,\mu)$ process studied in  Example \ref{selfsimilarESN} appears as the functional limit of certain Galton-Watson processes with immigration, see \cite{IKSANOV2018291}. They have shown that if $(Y_n,n\in \mathbb{N})$ is a Galton-Watson process with immigration (GWI) whose offspring distribution has finite mean $m$ and whose immigration distribution, say $\nu$, is such that $\bar{\nu}(n)\underset{n\rightarrow \infty}{\sim} \frac{c}{\log n}$ for some $c>0$ then
\[\left(\frac{1}{n}\big(\log Y_{[ns]}\big)_{+},s\geq 0\right)\underset{n\rightarrow \infty}{\Longrightarrow} (M(s),s\geq 0)\]
where $M$ is a selfsimilar standard ESN\footnote{The definition of the ESN process is slightly different in \cite{IKSANOV2018291}, see the forthcoming Remark \ref{remnopositivepart}} with $b=-\log m \in \mathbb{R}$ and  with $\bar{\mu}(x)=c/x$ for all $x>0$, the convergence holds in the Skorokhod sense and $[ns]$ denoted the integer part of $ns$. The analogue result for continuous-state branching processes with immigration, as well as other functional limit theorems, are established in \cite{foucartyuan2}. The form of the generator given in Theorem \ref{thmgeneratorESN} as well as their cores will play there a crucial role.

\end{remark}
\begin{example}[Stationary distributions of ESNs] Let $M$ be an $\mathrm{ESN}(b,\mu)$ process with $b=1$. Denote by $M(\infty)$ a random variable with law the stationary distribution $\pi$ whenever it exists.
\begin{enumerate}
\item Let $\alpha>1$ and assume $\bar{\mu}(x)=\frac{1}{x^{\alpha}}$ for all $x\geq 0$. Then $M$ does not hit $0$ a.s. (one has $\mathcal{J}=\infty$), is positive recurrent and $1/M(\infty)$ has a Weibull law with shape parameter $\alpha-1$, namely
\[\mathbb{P}(1/M(\infty)\geq y)=\pi([0,1/y])=e^{-\frac{1}{\alpha-1}y^{\alpha-1}}, \text{ for all } y\geq 0.\]  
\item Assume $\bar{\mu}(x)=e^{-x}$ for all $x\geq 0$. Then $M$ hits $0$ a.s. (one has $\bar{\mu}(0)<\infty$), is positive recurrent and $M(\infty)$ has the Gumbel law on $[0,\infty)$, namely
\[\mathbb{P}(M(\infty)\leq x)=\pi([0,x])=e^{-e^{-x}}, \text{ for all } x\geq 0.\]
\item  Assume $\bar{\mu}(x)=\frac{1}{x\log(1/x)}\mathbbm{1}_{\{x<1/e\}}$ for all $x\geq 0$. Then $M$ hits $0$ a.s. ($\bar{\mu}(0)=\infty$ but $\mathcal{J}<\infty$) is positive recurrent and
\[\mathbb{P}(M(\infty)\leq x)=\pi([0,x])=1/\log(1/x) , \text{ for all } 0\leq x\leq 1/e.\]
\end{enumerate}
\end{example}

\section{Study of ESN processes}\label{secproofESN}

We establish here the results of Section \ref{secESN}.
\subsection{Finite-dimensional laws, semigroup and stationary law of $\mathrm{ESNs}$: proof of Theorem \ref{thmfddsemigroupESN}}
\textbf{Proof of Theorem \ref{thmfddsemigroupESN}: (1)}. Recall $\mathcal{N}$ the Poisson point process with intensity $\lambda\times \mu$ and the Poisson construction of $M$ in Definition \ref{ESNdef}. Recall that almost surely for all $s\geq 0$, $M(s)\geq (-bs)_+$. Let $s_1>0$ and $u_1\in [0,\infty)$. The event $\{M(s_1)\leq u_1\}$ coincides almost surely with the event that all atoms $(t,\xi_t)$ of $\mathcal{N}$ on $[0,s_1]$ are such that $(\xi_t-b(s_1-t))_+\leq u_1$. Note that since $u_1\geq 0$, the inequality is equivalent to $\xi_t-b(s_1-t)\leq u_1$, and since any atom $\xi_t$ is positive, it is also equivalent to 
$\xi_t\leq (u_1+b(s_1-t))_+$, a.s..  More generally, for any $s_1<s_2<\cdots <s_n$ and $u_1,\dots, u_n\geq 0$,
\begin{align*}
&\left\{M(s_1)\leq u_1, M(s_2)\leq u_2, \cdots, M(s_n)\leq u_n\right\}\\
&=\left\{\forall t\in [0,s_1], \big(\xi_t-b(s_1-t)\big)_+\leq u_1, \forall t\in [0,s_2], \big(\xi_t-b(s_2-t)\big)_+\leq u_2, \cdots,\right. \\
& \qquad \qquad \qquad \left. \forall t\in [0,s_n], \big(\xi_t-b(s_n-t)\big)_+\leq u_n \right\}\cap \{u_1\geq (-bs_1)_+\}\cap \cdots\cap \{u_n\geq (-bs_n)_+\}\\
&=\big\{\mathcal{N}(A^c)=0\big\}\cap \{u_1\geq (-bs_1)_+\}\cap \cdots\cap  \{u_n\geq (-bs_n)_+\},
\end{align*}
with $A$, obtained by gathering all conditions on each disjoint intervals $(s_i,s_{i+1}]$, given by
\begin{align*}A:=&\left\{
0\leq t< s_1, \xi_t \leq \bigwedge_{i=1}^{n}\big(u_i+b(s_i-t)\big)_+, s_1\leq t< s_2, \xi_t\leq \bigwedge_{i=2}^{n}\big(u_i+b(s_i-t)\big)_+, \cdots ,\right.\\ 
&\qquad \left. s_{n-1}\leq t\leq s_n, \xi_t\leq \big(u_n +b(s_n-t)\big)_+ \right\}.
\end{align*}
Finally, since $\mathcal{N}(A^c)$ is a Poisson random variable with parameter $$(\lambda \times \mu)(A^c)=\sum_{i=1}^{n}\int_{s_{i-1}}^{s_{i}}\bar{\mu}\left(\bigwedge_{j=i}^{n}\Big(u_j+b(s_j-t)\Big)_+\right)\ddr t,$$ we get
\begin{align*}
\mathbb{P}_0(M(s_1)\leq &u_1,\cdots M(s_n)\leq u_n)\\
&=\mathbb{P}(\mathcal{N}(A^c)=0) \mathbbm{1}_{\{u_i\geq (-bs_i)_+, \forall 1\leq i\leq n\}}\\
&=\exp\left( -\sum_{i=1}^{n}\int_{s_{i-1}}^{s_{i}}\bar{\mu}\Big(\bigwedge_{j=i}^{n}\big(u_j+b(s_j-t)\big)\Big)\ddr t\right)\mathbbm{1}_{\{u_i\geq (-bs_i)_+, \forall 1\leq i\leq n\}}.
\end{align*}
The case $n=1, s_1=s,u_1=u$ gives \eqref{CDFM0allb}, namely 
\[F^0_s(u):=\mathbb{P}_0(M(s)\leq u)=\exp\left( -\int_{0}^{s}\bar{\mu}\left(u+b(s-t)\right)\ddr t\right)\mathbbm{1}_{\{u\geq (-bs)_+\}},\] 
and \eqref{CDFM0bneq0} is obtained by change of variable. \qed
\\

\noindent \textbf{Proof of Theorem \ref{thmfddsemigroupESN}: (2)}.
Recall that the ESN process started from $x$ is defined at any time $t$ by \eqref{Mx}, namely
%
$M^x(t)=M^0(t)\vee (x-bt)_+$. One has
\begin{align}\label{cdfMt}
    \mathbb{P}(M^x(t)\leq y)&=F^0_t(y)\mathbbm{1}_{\{y\geq (x-bt)_+\}}=F^0_t(y)\mathbbm{1}_{\{y\geq x-bt\}}.
\end{align}
Since the process $M$ takes only non-negative values,  $F^0_t(y)=0$ if $y<0$ and one can replace in \eqref{cdfMt}, the term $(x-bt)_+$ in the indicator function by $x-bt$. We shall use both writings. The expression in \eqref{cdfMt} is simpler to handle in some calculations. \qed\\

\noindent \textbf{Proof of Theorem \ref{thmfddsemigroupESN}: (3)}. The fact that $M$ satisfies the Markov property is checked as follows. Let $s,t\geq 0$ and $x\geq 0$, then
\begin{align*}
M^x(t+s)&=\sup_{0\leq u\leq t+s \atop \xi_0=x}\left(\xi_u-b(s+t-u)\right)_+\\
&=\sup_{0\leq u\leq t \atop \xi_0=x}\left(\xi_u-b(t-u)-bs\right)_+ \vee \sup_{t\leq u\leq t+s \atop \xi_0=x}\left(\xi_u-b(s+t-u)\right)_+\\
&=(M^x(t)-bs)_+\vee M(t,t+s),
\end{align*}
with $M(t,t+s):=\sup_{t\leq u\leq t+s}(\xi_{u}-b(t+s-u))_+=\sup_{0\leq u\leq s}(\xi_{u+t}-b(s-u))_+$ which is independent from $M^x(t)$. Note that $(M(t,t+s),s\geq 0)$ is a standard $\mathrm{ENS}(b,\mu)$ constructed from the PPP $\mathcal{N}$ shifted by time $t$.

We now check the Feller property. 
\begin{itemize}
\item[a)] Let $f\in C_0([0,\infty))$. We plainly see from \eqref{Mx} that almost surely for any $x_0>0$ and any $t\geq 0$, $M^{x}(t)\rightarrow M^{x_0}(t)$ as $x$ goes to $x_0$. For the case $x_0=0$, recall that $M^0(t)\geq -bt$ almost surely so that $(-bt)_+\vee M^0(t)=M^0(t)$. Therefore by continuity under expectation, the map $P_tf:x\mapsto \mathbb{E}[f(M^x(t)]$ is continuous on $[0,\infty)$. One also has $(x-bt)_+\vee M^0(t)\rightarrow \infty$ a.s. as $x$ goes $\infty$ and since $f$ is bounded, $P_tf(x)=\mathbb{E}\left(f\big((x-bt)_+\vee M^0(t)\big)\right)\underset{x\rightarrow \infty}{\longrightarrow} 0$.
\item[b)] For any $x\in [0,\infty)$, and all $t\geq 0$, by  \eqref{cdfMt} $$\mathbb{P}_x(M(t)\leq u)=\mathbb{P}_0(M(t)\leq u)\mathbbm{1}_{\{x-bt\leq u\}}.$$
   By \eqref{CDFM0allb}, $\mathbb{P}_0(M(t)\leq u){\rightarrow} 1$ as $t$ goes to $0$. Therefore, as $t$ goes to $0$, $\mathbb{P}_x(M(t)\leq u){\rightarrow} \mathbbm{1}_{\{x\leq u\}}$ and $M(t)$ converges in law towards $x$ under $\mathbb{P}_x$. This implies the pointwise continuity of the semigroup for given $x\geq 0$ as $t\to0$, 
   which is equivalent to the uniform one since $P_tC_0([0,\infty))\subset C_0([0,\infty))$, see e.g. \cite[Lemma 6.7 Chapter III]{zbMATH01478492}. \qed
\end{itemize}
\textbf{Proof of Theorem \ref{thmfddsemigroupESN}: (4) and (5)}.
Recall $b>0$. Since we assume \eqref{CDFM0bneq0}, for any $u\geq 0$,  \[F_s^0(u):=\mathbb{P}_0(M(s)\leq u)=\exp\left(-\frac{1}{b}\int_{u}^{u+bs}\bar{\mu}(y)\ddr y\right).\]
\begin{enumerate}
\item By letting $u$ go to $0$ in the expression above, we see that for all $s>0$, $\mathbb{P}_0(M(s)=0)=e^{-\frac{1}{b}\int_0^{bs}\bar{\mu}(v)\ddr v}$. The latter is strictly positive if and only if $\int_0^1 \bar{\mu}(v)\ddr v<\infty$. 
\item By letting $s$ go to $\infty$,  we see that
\[\underset{s\rightarrow \infty}{\lim} \mathbb{P}_0(M(s)\leq u)=\exp\left(-\frac{1}{b}\int_{u}^{\infty}\bar{\mu}(y)\ddr y\right)=\begin{cases} 0 & \text{if } \int_1^\infty \bar{\mu}(y)\ddr y=\infty\\
 >0 &\text{if } \int_1^\infty \bar{\mu}(y)\ddr y<\infty.
\end{cases}\]
Hence, the process converges towards $\infty$ in probability if and only if $\int_1^\infty \bar{\mu}(y)\ddr y=\infty$. Assume now $\int_1^\infty \bar{\mu}(y)\ddr y<\infty$ and let $\pi$ be a stationary distribution. Then for any $y\geq 0$,
\begin{align*} \pi([0,y])=\int_{0}^{\infty}\mathbb{P}_x(M(t)\leq y)\pi(\ddr x)&=\int_{0}^{\infty}\mathbb{P}_0(M(t)\leq y)\mathbbm{1}_{\{y\geq x-bt\}}\pi(\ddr x)\\
&=\mathbb{P}_0(M(t)\leq y)\pi([0,y+bt]).
\end{align*}
By letting $t$ to $\infty$, we see that $\pi$ exists if and only if $M$ admits a limiting distribution and that in this case $\pi$ and the latter coincide. Denote by $d_{\mathrm{TV}}$ the total variation distance and recall the coupling inequality, see e.g. \cite[Chapter 4, Lemma 4.1.11]{Zbl07762391}, 
\begin{equation}\label{dtv}
d_{\mathrm{TV}}\big(\mathbb{P}^{x}(M(t)\in \cdot),\pi(\cdot)\big)\leq \int_{0}^{\infty}\mathbb{P}(M^x(t)\neq M^y(t))\pi(\ddr y),
\end{equation}
with $M^x$ and $M^y$ the processes driven by the same Poisson point process $\mathcal{N}$ satisfying \eqref{Mx}. Note in particular that since $\pi$ is a stationary distribution,  $\int_0^\infty \mathbb{P}(M^y(t)\in \cdot)\pi(\ddr y)=\pi(\cdot)$ for all $t\geq 0$. By construction, if $bt>x\vee y$ then $M^{x}(t)=M^{y}(t)$ a.s.  Hence, when $bt>x$,
\[\int_{0}^{\infty}\mathbb{P}(M^x(t)\neq M^y(t))\pi(\ddr y)\leq \int_{bt}^{\infty}\mathbb{P}(M^x(t)\neq M^y(t))\pi(\ddr y)\leq \bar{\pi}(bt),\]
By the inequality $1-e^{-x}\leq x$, we have for any $t\geq 0$,
\[\bar{\pi}(bt)=1-\exp\left(-\frac{1}{b}\int_{bt}^{\infty}\bar{\mu}(u)\ddr u\right)\leq \frac{1}{b}\int_{bt}^{\infty}\bar{\mu}(u)\ddr u.\]
By plugging this upper bound in \eqref{dtv}, we get the inequality \eqref{dtvineq}. \qed
\end{enumerate} 

\begin{remark}\label{remnopositivepart} We have chosen here to work with \textit{nonnegative} extremal shot noise processes, see the positive parts in \eqref{ESNdef}. It is worth noticing however that if $M$ denotes a standard $\mathrm{ESN}(b,\mu)$ and $\bar{\mu}(0)=\infty$, then almost surely for all $t\geq 0$, \[M(t)=\sup_{0\leq s\leq t}(\xi_s-b(t-s))=:\tilde{M}(t).\] In other words, almost surely the process  $\tilde{M}$ defined above cannot take negative values. Indeed, a similar calculation as in the proof of Theorem \ref{thmfddsemigroupESN}-(1) when establishing \eqref{CDFM0allb}, would provide that for any $\epsilon>0$ and $s>0$, $\mathbb{P}(\tilde{M}(s)\leq -\epsilon)=0$. Since $\tilde{M}$ is c\`adl\`ag and $\epsilon$ is arbitrarily close to $0$, this entails 
$$\mathbb{P}(\exists s>0: \tilde{M}(s)<0)=
\mathbb{P}(\exists \epsilon \in \mathbb{Q}_+^\star, \exists s\in \mathbb{Q}_+^\star: \tilde{M}(s)<-\epsilon)=0,$$
where $\mathbb{Q}_+^\star$ is the set of positive rational numbers.
\end{remark}
\subsection{Infinitesimal generator of $\mathrm{ESNs}$: proof of Theorem \ref{thmgeneratorESN}}\label{secgen}
\begin{proof}[Proof of Theorem \ref{thmgeneratorESN}] 
Since the process is Feller, the generator $\mathcal{A}$  (obtained as the strong derivative of the semigroup) matches with the pointwise infinitesimal generator. In particular its domain is given by
\[\mathcal{D}(\mathcal{A})=\left\{f\in C_0([0,\infty): \exists g \in C_0([0,\infty))\ \forall x\in [0,\infty)\,:\, g(x)=\underset{t\rightarrow 0}{\lim} \frac{P_tf(x)-f(x)}{t}\right\}.\]
We refer e.g. to \cite[Theorem 1.33]{zbMATH06256582}. We shall therefore focus on pointwise convergence. 

Let $f$ be in $C^{1,0}([0,\infty))$. We see from \eqref{cdfMt} that for any $x\geq 0$, the semigroup of $M$ takes the form
\begin{equation}\label{semigroup} \mathbb{E}_x[f\big(M(t)\big)]=f\big((x-bt)_+\big)F_t^0\big((x-bt)_+\big)+\int_{\left((x-bt)_+,\infty\right)}f(y)\ddr F_t^0(y)\end{equation}
where $\ddr F_t^0$ denotes the Stieltjes measure associated to $F_t^0$ restricted on $(0,\infty)$, see \eqref{CDFM0allb}.

One has for all $x\geq 0$
    \begin{align}
        \frac{\mathbb{E}_x[f\big(M(t)\big)]-f(x)}{t}=
        &\frac{1}{t}\int_{(x-bt)_+}^{\infty}(f(y)-f(x))\ddr F^0_t(y) \label{jumppartA}\\
&+\frac{1}{t} \big(f\big((x-bt)_+\big)-f(x)\big)F_t^0\big((x-bt)_+\big). \label{driftpartA}     
        \end{align}
By \eqref{CDFM0allb}, for any $u>0$: \begin{equation}\label{weakconvA}\frac{1}{t}\left(1-F^0_t(u)\right)=\frac{1}{t}\mathbb{P}_0(M(t)>u)=\frac{1}{t}\left(1-e^{-\int_0^{t} \bar{\mu}\big(u+b(t-v)\big)\ddr v}\mathbbm{1}_{\{u\geq (-bt)_+\}}\right)\underset{t\rightarrow 0}{\rightarrow} \bar{\mu}(u).\end{equation}

For any $f\in \mathcal{D}_1$ and any $x\geq 0$, 
\begin{equation}\label{prelimitjumppart}
\int_x^{\infty}\big(f(y)-f(x)\big)\frac{\ddr F_t^{0}(y)}{t}=
\int_x^{\infty}f'(y)\frac{1-F_t^{0}(y)}{t}\ddr y.
\end{equation}
Using the inequality $1-e^{-u}\leq u$, we have for any slope $b\in \mathbb{R}$ and any $y\geq -bt$
\begin{align}
\frac{1-F^0_t(y)}{t}=
\frac{1}{t}\left(1-e^{-\int_0^t\bar{\mu}(y+b(t-v))\ddr v}\right)&\leq \frac{1}{t}\int_0^t\bar{\mu}(y+b(t-v))\ddr v\\
&\leq \bar{\mu}(y)\vee \bar{\mu}(y+bt). \label{bound1}
\end{align}
We study now the jump part \eqref{jumppartA} for $x>0$. By assumption $f$ is continuous and bounded, hence for all $x>0$, and for any small $\epsilon>0$, \[\underset{t\rightarrow 0+}{\limsup} \left|\frac{1}{t}\int_{(x-bt)_+}^{x}(f(y)-f(x))\ddr F^0_t(y)\right|\leq \sup_{y\in [x-\epsilon,x]}|f(y)-f(x)|\mu([x-\epsilon,x]),\]
which converges towards $0\times \mu(\{x\})=0$ as $\epsilon\to0$. Thus for all $x>0$ \begin{equation}\label{bordinteg}
\underset{t\rightarrow 0+}{\lim} \frac{1}{t}\int_{(x-bt)_+}^{x}(f(y)-f(x))\ddr F^0_t(y)=0.
\end{equation} 
Let $t_0>0$ be small enough such that $-bt_0<x$ then 
for any $y\geq x$ and $t\leq t_0$, $\bar{\mu}(y)\vee \bar{\mu}(y+bt)\leq \bar{\mu}(x)\vee \bar{\mu}(x+bt_0)$, hence by \eqref{bound1}
\begin{align}
\frac{1-F^0_t(y)}{t}&\leq\bar{\mu}(x)\vee \bar{\mu}(x+bt_0). \label{bound2}
\end{align}
Since $f\in C_{0}([0,\infty))$, $y\mapsto f'(y)$ is integrable on $(x,\infty)$ and by using the bound \eqref{bound2}, one can apply Lebesgue's theorem in \eqref{prelimitjumppart} to get for any $x>0$,
\begin{equation}\label{convf'}
\int_x^{\infty}f'(y)\frac{1-F_t^{0}(y)}{t}\ddr y\underset{t\rightarrow 0}{\longrightarrow} \int_x^{\infty}f'(y)\bar{\mu}(y)\ddr y.
\end{equation} 
Combining the convergences \eqref{convf'} with \eqref{bordinteg} and applying Fubini-Tonelli's theorem, we finally have for any $x>0$,
    \begin{equation}\label{convintegralpart}
    \underset{t\rightarrow 0+}{\lim} \frac{1}{t}\int_{(x-bt)_+}^{\infty}(f(y)-f(x))\ddr F^0_t(y)=\int_x^\infty(f(y)-f(x))\mu(\ddr y).
    \end{equation}    
We now deal with the second part \eqref{driftpartA}. For any $x>0$, we see from \eqref{CDFM0allb} that \[\mathbb{P}_0(M(t)\leq (x-bt)_+)=F^0_t((x-bt)_+)\underset{t\rightarrow 0}{\rightarrow} 1.\] Since $f$ is differentiable at $x$ we have that 
\[\underset{t\rightarrow 0+}{\lim} \frac{1}{t} (f\big((x-bt)_+\big)-f(x))F_t^0\big((x-bt)_+\big)=\underset{t\rightarrow 0+}{\lim} \frac{1}{t} \big(f(x-bt)-f(x)\big)=-bf'(x).\]
Hence, for any slope $b\in \mathbb{R}$, the generator $\mathcal{A}$ acts on any function $f\in C^{1,0}((0,\infty))$ at $x>0$ as follows:
\begin{equation}\label{formgenerator} \mathcal{A}f(x)=\int_x^{\infty}(f(y)-f(x))\mu(\ddr y)-bf'(x).
\end{equation}    
We are now going to study the convergence of $\frac{1}{t}\left(P_tf(0)-f(0)\right)$ as $t$ goes to $0$. Since the measure $\mu$ does not need to satisfy any integrability condition near $0$ a priori,  the right-hand side in \eqref{formgenerator} might not be well-defined when $x=0$ (even for instance if $f$ is $C^1$ at $0$). The generator at $x=0$ is thus more involved to study. Recall the sets $\mathcal{D}_0$ and $\mathcal{D}_1$ in \eqref{eqn:ddpositive} and \eqref{eqn:dd}. 

In the case $b\geq 0$, by \eqref{bound1}, 
\begin{equation}\label{dominationb>0}
\frac{1-F_t^{0}(y)}{t}\leq \bar{\mu}(y).\end{equation}
Since for any $f\in \mathcal{D}_1$, $\int_{0} |f'(y)|\bar{\mu}(y)\ddr y<\infty$, we get by applying Fubini-Tonelli's theorem and  Lebesgue's theorem, using the domination \eqref{dominationb>0},
\begin{align}
\frac{1}{t}\int_0^{\infty}(f(y)-f(0))\ddr F_t^{0}(y)=\int_0^{\infty}f'(y)\frac{1-F_t^{0}(y)}{t}\ddr y \underset{t\rightarrow 0}{\longrightarrow} 
\int_0^{\infty} \big(f(y)-f(0)\big)\mu(\ddr y).
\label{Af(0)welldefined}
\end{align}
When $b\geq 0$, the prelimit drift term \eqref{driftpartA} vanishes since $(0-bt)_+=0$, and for any $f\in \mathcal{D}_1$ 
\begin{equation}\label{convergencetoAf(0)}
\underset{t\rightarrow 0+}{\lim} \frac{1}{t}\left(P_tf(0)-f(0)\right)=\int_0^{\infty} \big(f(y)-f(0)\big)\mu(\ddr y)=:\mathcal{A}f(0).
\end{equation}
For $f$ to belong to $\mathcal{D}(\mathcal{A})$, $\mathcal{A}f$ should be continuous on $[0,\infty)$ and vanishing at $\infty$. The continuity on $(0,\infty)$ and the fact that $\mathcal{A}f(x)\underset{x\rightarrow \infty}{\longrightarrow} 0$ are clear since $f \in C^{1,0}([0,\infty))$. For the continuity at $0$, recall that by assumption if $f\in \mathcal{D}_1$, then $f'(0)=0$. We plainly see that $\mathcal{A}f(x)\underset{x\rightarrow 0}{\rightarrow}\mathcal{A}f(0)$. We have shown finally that in the case $b\geq 0$, $\mathcal{D}_1\subset \mathcal{D}(\mathcal{A})$. Recall also that $\mathcal{D}_0\subset\mathcal{D}_1$.

In the case $b<0$, we first check the convergence \eqref{Af(0)welldefined} for $f\in \mathcal{D}_0$. Let $\epsilon>0$ such that $f_{|[0,\epsilon]}$ is constant. By \eqref{convintegralpart}, one has
\[\frac{1}{t}\int_0^{\infty}(f(y)-f(0))\ddr F_t^{0}(y)=\frac{1}{t}\int_\epsilon^{\infty}(f(y)-f(0))\ddr F_t^{0}(y)\underset{t\rightarrow 0}{\longrightarrow} \int_0^{\infty}(f(y)-f(0))\mu(\ddr y).\]
Moreover the prelimit drift part \eqref{driftpartA} vanishes, hence the convergence \eqref{convergencetoAf(0)} holds true and $\mathcal{D}_0\subset \mathcal{D}(\mathcal{A})$. We now establish that $\mathcal{D}_1\subset \mathcal{D}(\mathcal{A})$. We shall use the following analytical lemma whose proof is postponed in the Appendix, see Section \ref{Appendix}.

\begin{lemma}[Approximation]\label{lemmaapprox} For any $f\in \mathcal{D}_1$, there is a sequence $(f_n)_{n\geq 1}$ in $\mathcal{D}_0$ such that 
\[f'_n\underset{n\rightarrow \infty}{\longrightarrow} f',\ f_n\underset{n\rightarrow \infty}{\longrightarrow} f \text{ uniformly and } |f'_n|\leq |f'| \text{ for all } n\geq 1.\]
\end{lemma}
\noindent Providing the sequence $(f_n)_{n\geq 1}$, we show that  $\mathcal{A}f_n\underset{n\rightarrow \infty}{\longrightarrow} g$ uniformly with $g$ given by \eqref{formgenerator}. For any $a>0$ and $x\geq 0$,
\begin{align*}
\lvert \mathcal{A}f_n(x)-g(x) \lvert \leq \int_a^{\infty}|f'_n-f'|(y) \bar{\mu}(y)\ddr y+\int_0^a|f'_n-f'|(y) \bar{\mu}(y)\ddr y+|b||f'(x)-f'_n(x)|.
\end{align*}
Since $|f'_n|\leq |f'|$ and $f'$ is integrable near $\infty$ with respect to $\bar{\mu}(y)\ddr y$ (note that $f'$ is integrable near $\infty$ since $f\in C_0((0,\infty))$ and $\bar{\mu}(y)\leq \bar{\mu}(a)$ for any $y\geq a$), by applying Lebesgue's theorem, we have that
\[\int_a^{\infty}|f'_n-f'|(y) \bar{\mu}(y)\ddr y\underset{n\rightarrow \infty}{\longrightarrow} 0.\]
By assumption $f\in \mathcal{D}_1$ and thus $f'$ is integrable near $0$ with respect to $\bar{\mu}(y)\ddr y$. As previously,  an application of Lebesgue's theorem, using the domination $|f'_n|\leq |f'|$, provides
\[\int_0^a|f'_n-f'|(y) \bar{\mu}(y)\ddr y \underset{n\rightarrow \infty}{\longrightarrow} 0.\]
Finally, the operator $\mathcal{A}$ being closed, see 
e.g. \cite[Lemma 19.8]{Kallenberg}, we have $f=\underset{n\rightarrow \infty}{\lim} f_n\in \mathcal{D}(\mathcal{A})$ and $\mathcal{A}f=\underset{n\rightarrow \infty}{\lim} \mathcal{A}f_n=g$. Hence $\mathcal{D}_1\subset \mathcal{D}(\mathcal{A})$ also when $b<0$.

We now study the core properties. A sufficient condition for a set of functions $\mathcal{D}\subset C_0([0,\infty))$ to be a core is that $\mathcal{D}$ is dense in $C_0([0,\infty))$ and  $P_t\mathcal{D} \subset \mathcal{D}$ for any $t\geq 0$. We refer for instance to \cite[Proposition 19.9]{Kallenberg}.

We treat first the case $b>0$. By \eqref{semigroup} for any function\footnote{This actually holds for any $f\in C_0([0,\infty))$} $f\in \mathcal{D}_1$, $P_tf(x)=P_tf(0)$ for $x\leq bt$. Hence the function $P_tf$ is constant near $0$.  
Since by assumption $f\in C^{1,0}([0,\infty))$, we see from \eqref{semigroup} and \eqref{CDFM0bneq0} that $x\mapsto P_tf(x)$ is differentiable on $[0,\infty)$ and for any $x\geq 0$
\[(P_tf)'(x)=\begin{cases}
f'(x-bt)F_t^{0}(x-bt)& \text{ if } x>bt\\
0& \text{ if } x\leq bt.
\end{cases}
\]
Plainly $P_tf$ and $(P_tf)'$ vanish at $\infty$ when $f\in \mathcal{D}_1$. Furthermore for any $t>0$, $(P_tf)'$ is continuous in $x$ for all $x>bt$ and all $x<bt$. For checking the continuity at $x=bt$, we use the assumption $f'(0)=0$ for any $f\in \mathcal{D}_1$, indeed 
\begin{equation}
 \label{continuityatbt}\underset{s\rightarrow 0 \atop s>0}{\lim} (P_tf)'(bt+s)=f'(0)F^0_t(0)=0=(P_tf)'(bt).
\end{equation}
Finally, we see that for any $t\geq 0$,
$P_t\mathcal{D}_0\subset P_t\mathcal{D}_1\subset \mathcal{D}_0\subset \mathcal{D}_1$ and $\mathcal{D}_0$ being dense in $C_0([0,\infty))$, both sets $\mathcal{D}_0$ and $\mathcal{D}_1$ are cores for $\mathcal{A}$ when $b>0$.  

We treat now the case $b\leq 0$. Recall also that we work under the assumption $\int_0 \bar{\mu}(x)\ddr x=\infty$. 
Similarly as before, we see from \eqref{semigroup} that for any $f\in C^{1,0}([0,\infty))$, $x\mapsto P_tf(x)$ is differentiable on $[0,\infty)$ and 
\[(P_tf)'(x)=f'(x-bt)F_t^{0}(x-bt).\]
This is clearly a continuous function on $[0,\infty)$ vanishing at $\infty$. Therefore $P_tf\in C^{1,0}([0,\infty))$.
We check now that $\int_{0}|(P_tf)'(x)|\bar{\mu}(x)\ddr x<\infty$. Note that this entails $(P_tf)'(0)=0$ since $\int_0 \bar{\mu}(x)\ddr x=\infty$. By \eqref{CDFM0allb}, we have for any $t\geq 0$ and $x\geq 0$, $F_t^0(x-bt)=\exp\left(-\int_0^t \bar{\mu}(x-br)\ddr r\right)$. 

In the case $b=0$, one has for some $\theta>0$ fixed and any $t>0$,
$$\int_{0}^{\theta}|(P_tf)'(x)|\bar{\mu}(x)\ddr x\leq ||f'||_\infty \int_0^{\theta} e^{-\bar{\mu}(x)t}\bar{\mu}(x)\ddr x\leq ||f'||_\infty \int_0^{\theta}\frac{\bar{\mu}(x)}{1+\bar{\mu}(x)}\ddr x<\infty.$$

In the case $b<0$, one has by \eqref{CDFM0bneq0}, for any $t>0$ and $x\geq 0$, $F_t^0(x-bt)=\exp\left(-\int_x^{x-bt} \bar{\mu}(y)\frac{\ddr y}{-b}\right)$ and thus for some  $0<\theta<-bt$ fixed:
\begin{align*}
\int_{0}^{\theta}|(P_tf)'(x)|\bar{\mu}(x)\ddr x&\leq ||f'||_\infty\int_0^{\theta} \exp\left(-\int_x^{x-bt}\bar{\mu}(y)\frac{\ddr y}{-b}\right)\bar{\mu}(x)\ddr x\\
&=-b||f'||_\infty \int_0^{\theta} \exp\left(-\int_x^{\theta}\bar{\mu}(y)\frac{\ddr y}{-b}\right)\exp\left(-\int_{\theta}^{x-bt}\bar{\mu}(y)\frac{\ddr y}{-b}\right)\frac{\bar{\mu}(x)}{-b}\ddr x\\
&\leq -b||f'||_\infty \exp\left(-\int_{\theta}^{-bt}\bar{\mu}(y)\frac{\ddr y}{-b}\right)\int_0^{\theta} \exp\left(-\int_x^{\theta}\bar{\mu}(y)\frac{\ddr y}{-b}\right)\frac{\bar{\mu}(x)}{-b}\ddr x\\
&=C\left[\exp\left(-\int_x^{\theta}\bar{\mu}(y)\frac{\ddr y}{-b}\right)\right]_{x=0}^{x=\theta}<\infty,
\end{align*}
with $C:=-b||f'||_\infty e^{-\int_{\theta}^{-bt}\bar{\mu}(y)\frac{\ddr y}{-b}}$.

Finally, $P_t\mathcal{D}_1 \subset \mathcal{D}_1$ for any $t\geq 0$ and $\mathcal{D}_1$ is a core for $\mathcal{A}$ when $b\leq 0$ and $\int_0 \bar{\mu}(x)\ddr x=\infty$. As mentioned in Remark \ref{remarksinglecore}, observe that under the assumption $\int_0^1\bar{\mu}(v)\ddr v=\infty$ (non-sticky case), the set $\mathcal{D}_1$ is a core for the $\mathrm{ESN}(b,\mu)$ for any $b\in \mathbb{R}$.
\end{proof}

\subsection{First passage times, transience, recurrence and zero set : proofs of Theorem \ref{thmfirstpassagetimeM} and Corollary \ref{corollaryrandomcutoutset}}
In all this section, we consider an $\mathrm{ESN}(b,\mu)$ with negative slopes, i.e. $b>0$. Notice that in the case $b\leq 0$, the process has almost-surely non-decreasing sample paths and questions to be addressed in this section are pointless.\\

\noindent \textbf{Proof of Theorem \ref{thmfirstpassagetimeM}-(1)}. Let $\theta>0$ and set for any $x>0$:
\begin{equation}\label{candidatethetainvariantfunction}
f_{\theta}(x):=\int_{x}^{\infty}e^{-\frac{\theta}{b}s}\exp{\left(\frac{1}{b}\int_{s}^{1}\bar{\mu}(u)\ddr u\right)}\ddr s.
\end{equation}
For any $x>0$, there is a constant  $C(x)>0$ such that $f_\theta(x)\leq C(x)\int_{x}^{\infty}e^{-\frac{\theta +\bar{\mu}(x)}{b}s}\ddr s<\infty$. Moreover, for all $x>0$, $f'_\theta(x)=-e^{-\frac{\theta }{b}x}\exp\left(\frac{1}{b}\int_x^1\bar{\mu}(u)\ddr u\right)$ and$f'_\theta \in C_{0}((0,\infty))$.
Recall the generator of $M$, $\mathcal{A}$ in \eqref{formgenerator}, and the form \eqref{alternativeform}. We verify now that $f_\theta$ is $\theta$-invariant for $\mathcal{A}$, i.e. $\mathcal{A}f_\theta=\theta f_\theta$.  One has for $x>0$,
\begin{align*}
    \mathcal{A}f_\theta(x)&=\int_x^{\infty}\bar{\mu}(v)f'_\theta(v)\ddr v -bf'_\theta(x)\\
    &=\int_{x}^{\infty}\big(-\bar{\mu}(v)\big)e^{-\frac{\theta}{b} v}\exp\left(\frac{1}{b}\int_v^1\bar{\mu}(u)\ddr u\right)\ddr v+be^{-\frac{\theta}{b}x}\exp\left(\frac{1}{b}\int_x^1\bar{\mu}(u) \ddr u\right)\\
    &=\left[be^{-\frac{\theta}{b}v}\exp\left(\frac{1}{b}\int_v^1\bar{\mu}(u)\ddr u\right)\right]_{v=x}^{\infty}+\int_{x}^{\infty}\theta e^{-\frac{\theta}{b} v}\exp\left(\frac{1}{b}\int_v^1\bar{\mu}(u)\ddr u\right)\ddr v\\
    &\qquad\qquad\qquad\qquad\qquad\qquad\qquad\qquad\qquad+be^{-\frac{\theta}{b}x}\exp\left(\frac{1}{b}\int_x^1\bar{\mu}(u) \ddr u\right)\\
    &=\theta \int_x^\infty e^{-\frac{\theta v}{b}}\exp\left(\frac{1}{b}\int_v^1\bar{\mu}(u)\ddr u\right)=\theta f_\theta(x),
\end{align*}
where in the third equality we have performed an integration by parts, together with the fact that $e^{-\frac{\theta}{b} v}\exp\left(-\frac{1}{b}\int_1^v\bar{\mu}(u)\ddr u \right)\underset{v\rightarrow \infty}{\rightarrow} 0$.

Let $\tilde f_\theta$ be a in $C^{1,0}([0,\infty))$ such that $\tilde f_\theta(v)=f_\theta(v)$ for any $ v\in [a/2,\infty)$ and $\tilde f_\theta$ is a constant on $[0,a/3]$. Then $\tilde f_\theta\in \mathcal{D}_0$, see \eqref{eqn:ddpositive}, and in particular by Theorem \ref{thmgeneratorESN} is in the domain of $\mathcal A$. Moreover, for any $x\geq a$, $\mathcal A \tilde{f}_\theta(x)=\mathcal A f_\theta(x)$.  Then by applying Dynkin's formula, see  \cite[(10.11), Chapter III.10, page 254]{zbMATH01478492} we have, for any $x>a$.  
\begin{align*}\E_x[e^{-\theta \sigma_a\wedge t}f_\theta(M(\sigma_a\wedge t))]- f_\theta(x)&=\E_x[e^{-\theta \sigma_a\wedge t} \tilde f_\theta(M(\sigma_a\wedge t))]-\tilde f_\theta(x)\\
&=\E_x\left[\int_0^{\sigma_a\wedge t} e^{-\theta s}(\mathcal{A}\tilde f_\theta-\theta \tilde f_{\theta})(M(s))\ddr s\right]\\
&=\E_x\left[\int_0^{\sigma_a\wedge t} e^{-\theta s}(\mathcal{A} f_\theta-\theta  f_{\theta})(M(s))\ddr s\right]=0.\end{align*}
Since the process $M$  has no negative jumps, one has almost surely $M(\sigma_a)=a$ on the event $\{\sigma_a<\infty\}$. The function $f_\theta$ being continuous, we get by letting $t$ go to $\infty$:
\begin{equation}\label{equationthetainvariantfunc} \mathbb{E}_x[e^{-\theta \sigma_a}\mathbbm{1}_{\{\sigma_a<\infty\}}]= \mathbb{E}_x[e^{-\theta \sigma_a}]=f_\theta(x)/f_\theta(a).\end{equation}
\qed
\\
\textbf{Proof of Theorem \ref{thmfirstpassagetimeM}-2).}
We now study the recurrence and the transience of the process. 

Recall 
$\mathcal{I}:=\int^{\infty}_1\exp\left( \frac{1}{b}\int_{s}^1\bar{\mu}(v)\ddr v\right)\ddr s$ and
define the function $g(s):=\exp\left( \frac{1}{b}\int_{s}^1\bar{\mu}(v)\ddr v\right)$ for all $s>0$ so that $\mathcal{I}=\int_1^{\infty}g(s)\ddr s$ and $f_\theta(x)=\int_x^{\infty}e^{-\frac{\theta}{b}s}g(s)\ddr s$ for all $x>0$. Moreover one has $\int_a^{1}g(s)\ddr s\leq e^{\frac{\theta}{b}}f_\theta(a)<\infty$ for all $a>0$ and by  \eqref{equationthetainvariantfunc}:
\begin{align}
\mathbb{E}_x[e^{-\theta \sigma_a}]=\frac{f_\theta(x)}{f_\theta(a)}&=\frac{\int_x^1 e^{-\frac{\theta}{b}s}g(s)\ddr s+ \int_1^\infty e^{-\frac{\theta}{b}s}g(s)\ddr s}{\int_a^1 e^{-\frac{\theta}{b}s}g(s)\ddr s+ \int_1^\infty e^{-\frac{\theta}{b}s}g(s)\ddr s}\nonumber \\
&=\frac{\frac{\int_x^1 e^{-\frac{\theta}{b}s}g(s)\ddr s}{\int_1^\infty e^{-\frac{\theta}{b}s}g(s)\ddr s}+ 1}{\frac{\int_a^1 e^{-\frac{\theta}{b}s}g(s)\ddr s}{\int_1^\infty e^{-\frac{\theta}{b}s}g(s)\ddr s}+1}.\label{fraction}
\end{align}
\begin{enumerate}
\item Recurrence: assume  $\mathcal{I}=\infty$.  One has $\int_1^{\infty}g(s)\ddr s=\infty$ and by monotone convergence, we have
\[f_\theta(1)=\int_1^{\infty}e^{-\frac{\theta}{b}s}g(s)\ddr s\underset{\theta \rightarrow 0^+}{\rightarrow} \mathcal{I}=\infty,\]
and $\int_a^1 e^{-\frac{\theta}{b}s}g(s)\ddr s\underset{\theta \rightarrow 0^+}{\rightarrow} \int_a^1 g(s)\ddr s<\infty$. Hence, by letting $\theta$ go to $0$ in \eqref{fraction}, we get 
\[\mathbb{P}_x(\sigma_a<\infty)=\underset{\theta \rightarrow 0^+}{\lim} \mathbb{E}_x[e^{-\theta \sigma_a}]= \underset{\theta \rightarrow 0^+}{\lim}\frac{f_\theta(x)}{f_\theta(a)}=1.\]
\item Transience: assume $\mathcal{I}<\infty$.  One has $\int_1^{\infty}g(s)\ddr s<\infty$. By letting $\theta$ go to $0$ in \eqref{fraction}, we see that 
\begin{equation}\label{probahitting}
\mathbb{P}_x(\sigma_a<\infty)=\underset{\theta \rightarrow 0+}{\lim}\frac{f_\theta(x)}{f_\theta(a)}<1.
\end{equation}
It remains to show that the process $M$ goes to $\infty$ a.s.. Denote by $\theta_t$ the time shift operator, i.e. $\theta_t(M(\cdot))=M(t+\cdot)$. For any $a>0$, one has
\begin{align}\label{transiencedecompose}
\mathbb{P}_x(\underset{t\rightarrow \infty}{\liminf}\, M(t)<a)&\leq \underset{t\rightarrow \infty}{\liminf}\,\mathbb{P}_x(\sigma_a\circ \theta_t<\infty)\nonumber \\
&=\underset{t\rightarrow \infty}{\liminf}\,\mathbb{E}_x\left[\mathbb{P}_{M(t)}(\sigma_a<\infty)\right].
\end{align}
Moreover
\begin{align*}
\mathbb{E}_x\left[\mathbb{P}_{M(t)}(\sigma_a<\infty)\right]&\leq \mathbb{P}_{x}(M(t)\leq a)+\mathbb{E}_x\left[\mathbbm{1}_{\{M(t)>a\}}\mathbb{P}_{M(t)}(\sigma_a<\infty)\right]\\
&=:\mathrm{I}+\mathrm{II}.
\end{align*}
Note that the condition $\mathcal{I}<\infty$ implies that $\int^{\infty}\bar{\mu}(u)\ddr u=\infty$ and by Theorem \ref{thmfddsemigroupESN}-(3), in this case the process goes to $\infty$ in probability. 
The first term $\mathrm{I}$ at the right-hand side in \eqref{transiencedecompose} goes towards $0$ as $t$ goes to $\infty$. We study now the second term $\mathrm{II}$. By the assumption $\mathcal{I}<\infty$, we see that for any $a>0$, $\theta\mapsto f_\theta(a)$ is well-defined and continuous at $\theta=0$. Recall $f_0(x)=\int_x^\infty g(s)\ddr s$ for all $x\geq 0$. One has by \eqref{probahitting} and  by applying Fubini-Tonelli theorem in the last inequality,
\begin{align*}
\mathbb{E}_x[\mathbbm{1}_{\{M(t)>a\}}\mathbb{P}_{M(t)}(\sigma_a<\infty)]&=\mathbb{E}_x\left[\mathbbm{1}_{\{M(t)>a\}}\frac{f_0\big(M(t)\big)}{f_0(a)}\right]\\
&=\frac{1}{f_0(a)}\mathbb{E}_x \left[\int_{0}^\infty \mathbbm{1}_{\{a<M(t)\leq s\}}g(s)\ddr s\right]\\
&\leq \frac{1}{f_0(a)}\int_{a}^{\infty}\mathbb{P}_x(M(t)\leq s)g(s)\ddr s.
\end{align*}
By dominated convergence, since $\mathbb{P}_x(M(t)\leq s)\underset{t\rightarrow \infty}{\longrightarrow} 0$, we have that 
\[\mathrm{II}:=\mathbb{E}_x[\mathbbm{1}_{\{M(t)>a\}}\mathbb{P}_{M(t)}(\sigma_a<\infty)]\underset{t\rightarrow \infty}{\longrightarrow} 0.\]
Therefore 
$\mathbb{P}_x(\underset{t\rightarrow \infty}{\liminf}\, M(t)<a)=0$ and since $a$ can be arbitrarily large, $\underset{t\rightarrow \infty}{\liminf}\, M(t)=\infty$ a.s. \qed
\end{enumerate}
\textbf{Proof of Theorem \ref{thmfirstpassagetimeM}-3)}. Recall $\mathcal{J}=\int_{0}^1\exp\left(\frac{1}{b}\int_{s}^1\bar{\mu}(v)\ddr v\right)\ddr s$. Given an initial value $x>0$, for any $x>a_1>a_2>0$, we have that $\sigma_{a_1}\leq \sigma_{a_2}\leq  \sigma_0$ a.s.. Therefore $\sigma_{0^+}:=\underset{a\rightarrow 0^+}{\lim}\uparrow \sigma_a\leq \sigma_0$ a.s.. Since the process $M$ is Feller, it is quasi-continuous to the left and one has by the absence of negative jumps: on the event $\{\sigma_{0^+}<\infty\}$:
\[M(\sigma_{0^+})=\underset{a\rightarrow 0^+}{\lim} M(\sigma_a)=\underset{a\rightarrow 0^+}{\lim} a=0.\] 
Thus, since by definition $\sigma_0$ is the first hitting time of $0$, $\sigma_{0^+}\geq \sigma_0$ a.s. This entails $\sigma_{0^+}=\sigma_0$. On the event $\sigma_{0^+}=\infty$, trivially $\sigma_0=\infty$. To sum up, we have $\sigma_{0+}=\sigma_0$ a.s.  By \eqref{equationthetainvariantfunc}, for any $\theta>0$, $\mathbb{E}_x[e^{-\theta \sigma_a}]=\frac{f_\theta(x)}{f_\theta(a)}$. In this equality, by letting $a$ go to $0$, we see that  
$$\mathbb{E}_x[e^{-\theta \sigma_0}]=\frac{f_\theta(x)}{f_\theta(0)}, $$
with \begin{align*}f_\theta(0):=f_\theta(0^+)&=\int_{0}^{\infty}e^{-\frac{\theta}{b}s}\exp\left(\int_s^1\bar{\mu}(u)\ddr u\right)\ddr s\\
&=\int_{0}^{1}e^{-\frac{\theta}{b}s}\exp\left(\int_s^1\bar{\mu}(u)\ddr u\right)\ddr s+ \int_{1}^{\infty}e^{-\frac{\theta}{b}s}\exp\left(\int_s^1\bar{\mu}(u)\ddr u\right)\ddr s.
\end{align*}
The second term on the right-hand side is nothing but $f_\theta(1)$ which is always finite. The first term is finite if and only if $\mathcal{J}<\infty$. In the case $\mathcal{J}=\infty$, one therefore has
$\mathbb{E}_x[e^{-\theta \sigma_0}]=0$ and $\sigma_0=\infty$ a.s. Otherwise, when $\mathcal{J}<\infty$, we have $f_\theta(0)<\infty$ and $\sigma_0<\infty$ with positive probability. 

The proof of Theorem \ref{thmfirstpassagetimeM} is achieved. \qed

\begin{remark} 
The function $f_\theta$ is not $\theta$-invariant for the semigroup $(P_t)$, but only $\theta$-excessive, namely $e^{-\theta t}P_t f_\theta \leq f_\theta$ for all $t\geq 0$. We show indeed below that for all $x$ and $t$:
\begin{equation}\label{meanftheta} P_tf_\theta(x)=\mathbb{E}_x[f_\theta\big(M(t)\big)]=e^{\theta t}f_\theta(x\vee bt).
\end{equation}
By Fubini-Tonelli's theorem and the expression of $\mathbb{P}_x(M(t)\leq s)$ given in Theorem \ref{thmfddsemigroupESN}-(2), we get
\begin{align*}
\mathbb{E}_x[f_\theta\big(M(t)\big)]&=\int_{0}^{\infty}\mathbb{E}_x[\mathbbm{1}_{\{M(t)\leq s\}}]e^{-\frac{\theta}{b} s}\exp\left(\frac{1}{b}\int_{s}^{1}\bar{\mu}(u)\ddr u\right)\ddr s\\
&=\int_{0}^{\infty}\exp\left(-\frac{1}{b}\int_{s}^{s+bt}\bar{\mu}(u)\ddr u \right)\mathbbm{1}_{\{s\geq (x-bt)_+\}}\exp\left(\frac{1}{b}\int_{s}^{1}\bar{\mu}(u)\ddr u\right)e^{-\frac{\theta}{b} s}\ddr s\\
&=\int_{(x-bt)_+}^\infty\exp\left(-\frac{1}{b}\int_1^{s+bt}\bar{\mu}(u)\ddr u \right)e^{-\frac{\theta}{b} s}\ddr s. 
\end{align*}
The change of variable $v=s+bt$ provides
\begin{align*}
\mathbb{E}_x[f_\theta\big(M(t)\big)]&=\int_{(x-bt)_+ +bt}^\infty\exp\left(-\frac{1}{b}\int_1^{v}\bar{\mu}(u)\ddr u \right)e^{-\theta \frac{v}{b}}e^{\theta t}\ddr v\\
&=e^{\theta t}f_\theta(bt)\mathbbm{1}_{\{x<bt\}}+ e^{\theta t}f_\theta(x)\mathbbm{1}_{\{x\geq bt\}}=e^{\theta t}f_\theta(x\vee bt) \leq e^{\theta t}f_\theta(x),
\end{align*}
since $f_\theta$ is decreasing. We see that when $x<bt$, $e^{-\theta t}P_t f_\theta(x) \neq f_\theta(x)$. Hence $f_\theta$ is not $\theta$-invariant for the semigroup, although we have seen in the proof of Theorem \ref{thmfirstpassagetimeM}-(1), that $\mathscr{A}f_\theta=\theta f_\theta$. The \textit{unstopped} process $(e^{-\theta t}f_\theta\big(M(t)\big),t\geq 0)$ is therefore a strict local martingale and $f_\theta$ does not belong to the domain of the ESN process.
\end{remark}

It remains to identify the law of the inverse local time of $M$ when $0$ is accessible, i.e. when $\mathcal{J}<\infty$. Recall that this entails $f_\theta(0)<\infty$ for all $\theta>0$.\\

\noindent \textbf{Proof of Theorem \ref{thminverselocaltime}}. The fact that $0$ is regular for itself is shown by the following simple argument: under the assumption $\mathcal{J}<\infty$, for any $\theta>0$, the function $f_\theta$ is well-defined and continuous at $0$, hence
$\mathbb{E}_x(e^{-\theta \sigma_0})=f_\theta(x)/f_\theta(0)\longrightarrow 1$ as $x$ goes to $0$, so that, for any $t>0$,
\begin{equation}\label{eqregularforitself}
\mathbb{P}_x(\sigma_0>t)\underset{x\rightarrow 0}{\longrightarrow} 0.\end{equation}
The above display will entail the regularity of $0$ as we show below.   

Denote by $\mathrm{r}_0 :=\inf\{t>0: M(t)=0\}$, the first return time to $0$. Recall that by definition $0$ is regular for itself if the process returns immediately to $0$ after it has left it, that is to say if $\mathbb{P}_0(\mathrm{r}_0 =0)=1$. Notice that for any $x>0$, $\mathrm{r}_0 =\sigma_0$ a.s. under $\mathbb{P}_x$. For any $s,t>0$. By the Markov property at a time $s$, and by using the fact that $M(s) \underset{s\rightarrow 0+}{\longrightarrow} 0$ a.s. under $\mathbb{P}_0$ together with \eqref{eqregularforitself}, one has by Lebesgue's theorem
\[\mathbb{P}_0(\mathrm{r}_0 >t+s)=\mathbb{E}_0\left[\mathbb{P}_{M(s)}(\sigma_0>t)\mathbbm{1}_{\{\mathrm{r}_0 >s\}}\right]\underset{s\rightarrow 0+}{\longrightarrow} 0.\]
Therefore $\mathbb{P}_0(\mathrm{r}_0 >t)=0$ for any $t>0$ and $0$ is regular for itself. 

We now apply general results about local times. Recall that the local time  $(L_t,t\geq 0)$ at $0$ of $M$ is uniquely defined up to a multiplicative positive constant.  \cite{zbMATH03205726} have shown how the $\theta$-potential operators of the local time can be associated to a family of $\theta$-excessive functions, see the functions $(\Psi_1^\theta,\theta>0)$ below, and how one can relate the Laplace exponent of the inverse local time to $(\Psi_1^\theta,\theta>0)$. Following their notation, we introduce, for any $\theta>0$,
\[\Phi^{\theta}:x\mapsto \mathbb{E}_x[e^{-\theta \sigma_0}] \text{ and } \Psi_1^{\theta}:=\Phi^{1}-(\theta-1)U^{\theta}\Phi^{1}, \]
with $U^{\theta}$ the resolvent of the process $M$, that is for any bounded measurable function $f$
\[U^{\theta}f(x):=\mathbb{E}_x\left[\int_0^{\infty}e^{-\theta t}f\big(M(t)\big)\ddr t\right].\]
By \cite[Theorem 1.2]{zbMATH03205726}, which applies since $0$ is regular for itself, the local time of $M$ at $0$ is the unique process $L:=(L_t,t\geq 0)$ such that for all $x\geq 0$ and all $\theta>0$,
\[\mathbb{E}_x\left[\int_0^\infty e^{-\theta t}\ddr L_t\right]=\Psi_1^{\theta}(x).\]
By applying \cite[Theorem 2.1]{zbMATH03205726}, with their notation $x=x_0=0$, the inverse of the local time is a subordinator, that we denote by $(\tau_x, 0\leq x< \zeta)$, with Laplace exponent $\varphi(\theta):=1/\Psi_1^\theta(0)$. 
Only remains to compute $\Psi_1^\theta(0)=1-(\theta-1)U^{\theta}\Phi^1(0)$. By Theorem \ref{thmfirstpassagetimeM}-3), for any $x\geq 0$, $\Phi^1(x)=\frac{f_1(x)}{f_1(0)}$. Moreover, making use of \eqref{meanftheta}, we have
\begin{align*}
U^{\theta}f_1(0)&=\int_0^{\infty}e^{-\theta t}\mathbb{E}_0[f_1\big(M(t)\big)]\ddr t=\int_0^{\infty}e^{-(\theta-1) t}f_1(bt)\ddr t,
\end{align*}
and thus
\[\Psi_1^{\theta}(0)=1-(\theta-1)\int_{0}^{\infty}e^{-(\theta-1)t}\frac{f_1(bt)}{f_1(0)}\ddr t.\]
Recall $f_\theta$ in \eqref{thetainvariantfunction} and note that $f'_1(bt)=-e^{-t}\exp\left(\frac{1}{b}\int_{bt}^{1}\bar{\mu}(u)\ddr u\right)$ for all $t\geq 0$. By integration by parts, we have
\begin{align*}
\int_{0}^{\infty}(\theta-1)e^{-(\theta-1)t}f_1(bt)\ddr t&=\left[-e^{-(\theta-1)t}f_1(bt)\right]_{t=0}^{t=\infty}+\int_{0}^{\infty}e^{-(\theta-1)t}bf'_1(bt)\ddr t\\
&=f_1(0)-\int_{0}^{\infty}e^{-(\theta-1)t}be^{-t}\exp\left(\frac{1}{b}\int_{bt}^{1}\bar{\mu}(u)\ddr u\right)\ddr t\\
&=f_1(0)-\int_{0}^{\infty}be^{-\theta t}\exp\left(\frac{1}{b}\int_{bt}^{1}\bar{\mu}(u)\ddr u\right)\ddr t\\
&=f_1(0)-f_\theta(0),
\end{align*}
where for the last equality we perform the change of variable $s=bt$. Therefore
\[\Psi_1^{\theta}(0)=1-\frac{f_1(0)-f_\theta(0)}{f_1(0)}=f_\theta(0)/f_1(0),\]
and the Laplace exponent $\varphi$ of $(\tau_x, x<\zeta)$ is given by $\varphi(\theta)=1/\Psi_1^{\theta}(0)=f_1(0)/f_\theta(0)$. 

Last, the fact that the closure of the zero-set of $M$ started from $0$ is the closed range of the subordinator $(\tau_x,0\leq x<\zeta)$, that is to say
\[\overline{\{t\geq 0: M(t)=0\}}=\overline{\{\tau_x: 0\leq x<\zeta\}} \text{ a.s.}\]
follows from a standard result on local times, see e.g. \cite[Chapter IV, Theorem 4-(iii)]{Bertoin96}.
\qed

We now explain the connection between the ESN processes and the random cutout sets by establishing Lemma \ref{lemmazeroset}.
\medskip

\noindent \textbf{Proof of Lemma \ref{lemmazeroset}}. Recall $b>0$. Recall Definition \ref{ESNdef}. The $\mathrm{ESN}(b,\mu)$ started at $0$ is given for all $t\geq 0$ by $M(t)=\sup_{0\leq s\leq t}(\xi_s-b(t-s))_+$. We see that $M(t)>0$ if and only if there exists an atom $(s,\xi_s)$ such that $s\leq t$ and $\xi_s-b(t-s)>0$, namely $t\in \bigcup_{s\geq 0}[s,s+\xi_s/b)$. Recall that $M(t)\geq 0$ a.s. Therefore $\{M(t)>0\}^c=\{M(t)=0\}$ for all $t$ a.s. and 
\[\mathcal{Z}=\{t\geq 0: M(t)=0\}=[0,\infty[-\bigcup_{s\geq 0}[s,s+\xi_s/b).\]
Plainly, $\bar{\mathcal{Z}}:=\overline{\{t>0: M(t)=0\}}\subset [0,\infty[-\bigcup_{s\geq 0}(s,s+\xi_s/b)$. We now show the other inclusion.
For any time $t\in \bar{\mathcal{Z}}-\mathcal{Z}$, there exists a sequence $(t_n)_{n\geq 1}$ such that, for all $n$, $t_n\in \mathcal{Z}$, $t_n<t$ and $t_n\underset{n\rightarrow \infty}{\rightarrow} t$. Hence, $M$ has zero for left limit at time $t$, $M(t-)=0$. By construction the only possibility is that $t$ is an atom of time, say $u$, of the Poisson point process $\mathcal{N}$ such that $M(u-)=0$. We now verify that all such atoms belong to $[0,\infty[-\bigcup_{s\geq 0}(s,s+\xi_s/b)$. By contradiction, suppose that $u\in \bigcup_{s\geq 0}(s,s+\xi_s/b)$. There exists then $s$ such that for all $n$  large enough $t_n\in (s,s+b\xi_s)$. This would entail $M(t_n)>0$ which is not possible since $t_n\in \mathcal{Z}$. 
%
%
\qed
 
It only remains to explain how Corollary \ref{corollaryrandomcutoutset}, which restates Fitzsimmons-Fristedt-Shepp Theorem,  is deduced.
\medskip

\noindent \textbf{Proof of Corollary \ref{corollaryrandomcutoutset}}. It will follow directly by applying Lemma \ref{lemmazeroset}, Theorem \ref{thmfirstpassagetimeM} and Theorem \ref{thmfddsemigroupESN}-3. in the case $b=1$. Consider $M$ the $\mathrm{ESN}(1,\mu)$ started at $0$, constructed from the same Poisson point process $\mathcal{N}$ as the random cutout set $\mathcal{R}$. By Lemma \ref{lemmazeroset}, one has the identity
$$\mathcal{R}=\overline{\{t\geq 0 : M(t)=0\}}.$$
In particular, by Theorem \ref{thmfirstpassagetimeM}-3, $0$ is inaccessible for $M$, equivalently $\mathcal{R}=\{0\}$ a.s. if and only if $\mathcal{J}=\infty$ (with $b=1$). When $\mathcal{J}<\infty$, $\mathcal{R}$ coincides with $\bar{\mathcal{Z}}=\overline{\{\tau_x: 0\leq x< \zeta\}}$ 
where $(\tau_x, 0\leq x< \zeta)$ the subordinator defined as the inverse local time of $M$ at $0$. By Theorem \ref{thminverselocaltime}, its Laplace exponent is $\varphi:\theta\mapsto f_1(0)/f_\theta(0)$. 


The fact that $\mathcal{R}$ is a regenerative set is immediate since $\bar{\mathcal{Z}}$ is the closure of the range of the subordinator $(\tau_x,0\leq x<\zeta)$, see \cite[Chapter 2.1]{subordinatorbertoin}. It is perfect a.s. since the boundary $0$ of $M$ is regular for itself. The set $\mathcal{R}$ is bounded a.s. if and only if $M$ is transient, namely $\mathcal{I}<\infty$.  We have seen in Theorem \ref{thmfddsemigroupESN}-3, that $\mathbb{P}(t\in \mathcal{Z})>0$ for all $t>0$ if and only if $\int_0^1\bar{\mu}(s)\ddr s<\infty$. Fix $a>0$, an application of Fubini-Tonelli theorem yields  $\int_0^a\mathbb{P}_0(M(s)=0)\ddr s=\mathbb{E}[\lambda(\mathcal{Z}\cap [0,a])]>0$ (which is equivalent to $\lambda(\mathcal{Z}\cap [0,a])>0$ with positive probability) if and only if $\int_0^1 \bar{\mu}(v)\ddr v<\infty$. Recall that $\lambda(\mathcal{Z})$ is either zero almost surely or positive almost surely, see e.g. \cite[Proposition 1.8]{subordinatorbertoin}. Then we can conclude that $\lambda(\mathcal{Z})$ is positive almost surely when $\int_0^1 \bar{\mu}(v)\ddr v<\infty$. 
\qed
\begin{remark}
Our proof of Fitzsimmons-Fristedt-Shepp Theorem is different from those already given in the literature, see \cite{Fitzsimmons}, \cite{subordinatorbertoin}, and Fitzsimmons's PhD thesis \cite[Chapter 4]{Fitzthesis}, as well as \cite{zbMATH04203384} where the point of view of random multiplicative measures is chosen. These proofs are based on approximations of the random set $\mathcal{R}$, see \eqref{randomcutoutset}, through random cutout sets $\mathcal{R}^{(\epsilon)}$ with finite measures $\mu^{(\epsilon)}$ such that $\mu^{(\epsilon)}$ converges towards $\mu$ as $\epsilon$ goes to $0$, and on limit theorems for regenerative sets, see \cite{zbMATH03874373} and \cite[Lemma 3, Chapter 4]{Fitzthesis}. Here we do not approximate $\mathcal{R}$, the uncovered points, i.e. the elements of $\mathcal{R}$, are seen as zeros (or limits of zeros) of the standard ESN process $M$, and are encoded by its local time. In other words, the covered set is viewed as the union of intervals of excursions away from $0$ of the Markov process $M$. 
The core of the proof lies in the fact that we have  at hand a $\theta$-invariant function $f_\theta$, see \eqref{thetainvariantfunction}, of the generator $\mathcal{A}$. 
\end{remark}

\begin{acks}[Acknowledgments]
We thank Patrick Fitzsimmons for a discussion  after finishing this work about the initial proof of \cite[Theorem 1]{Fitzsimmons} and for sending an unpublished part of his PhD thesis. Lemma \ref{lemmazeroset} in particular has been observed in \cite[Chapter 4, Section 3]{Fitzthesis}.  We also wish to thank the anonymous referees for their thorough reports. We are grateful for a comment leading to a neat result on the speed of the convergence towards the limiting distribution in total variation distance. 
\end{acks}
\section{Appendix}\label{Appendix}

\begin{proof}[Proof of Lemma \ref{lemmaapprox}] Let $f\in \mathcal{D}_1$. We shall define the sequence $(f_n)_{n\geq 1}$ in $\mathcal{D}_0$ approaching the function $f$ and satisfying $|f'_n|\leq |f'|$ within two steps. Recall that $f'(0)=0$. Define $g_n$ as follows: for any $y\in [0,1/n]$, $g_n(y):=0$ then for $y\in[1/n,2/n]$, $g_n(y):=\frac{y-1/n}{1/n}f'(2/n)$ and $g_n(y):=f'(y)$ for $y\geq 2/n$. One has 
\[\sup_{y\in [0,\infty)}|g_n(y)-f'(y)|\leq \sup_{0\leq y\leq 1/n}|f'(y)|+\sup_{1/n\leq y\leq 2/n}\left|\frac{y-1/n}{1/n}f'(2/n)-f'(y)\right|\]
Since $f'$ is continuous and $f'(0)=0$ the first term at the right-hand side above goes to $0$ as $n$ goes to $\infty$. The second term is smaller than $|f'(2/n)|+ \sup_{0\leq y\leq 2/n}|f'(y)|$ which goes to $0$ as $n$ goes to $\infty$. Hence $g_n$ tends uniformly towards $f'$. 

Set now \[\bar{g}_n(y):=\begin{cases} g_n(y)  &\text{ if } |g_n(y)|\leq |f'(y)|\\
f'(y)  &\text{ if } |g_n(y)|>|f'(y)|.
 \end{cases}\]
Notice that $|\bar{g}_n|=|g_n|\wedge |f'|\leq |f'|$ and
$\bar{g}_n $ is continuous. One can plainly check furthermore that $\bar{g}_n \underset{n\rightarrow \infty}{\longrightarrow} f'$ uniformly. Define $f_n(x):=f(0)+\int_0^x \bar{g}_n(y)\ddr y$ for any $x\in [0,\infty)$. The function $f_n$ is constant on $[0,1/n]$, hence $f_n\in \mathcal{D}_0$ and $|f'_n|=|\bar{g}_n|\leq |f'|$ for all $n$. Last, for any $x\geq 0$,
\[\lvert f_n(x)-f(x) \lvert \leq \int_0^x \lvert \bar{g}_n(y)-f'(y)\lvert \ddr y\leq \frac{2}{n}||\bar{g}_n-f'||_\infty \underset{n\rightarrow \infty}{\longrightarrow} 0.\]
\end{proof}

{
\begin{funding}
The authors are supported  by the French National Research Agency (ANR): LABEX MME-DII (ANR11-LBX-0023-01). C.F is also supported by the European Union (ERC, SINGER, 101054787). Views and opinions expressed are however those of the authors only and do not necessarily reflect those of the European Union or the European Research Council. Neither the European Union nor the granting authority can be held responsible for them. %
\end{funding}

\end{document}